\documentclass[aip,graphicx,amsmath,amssymb]{revtex4-1}
\usepackage{amsmath,amsthm,amsfonts,amssymb,layout,indentfirst}
\usepackage{graphics,epsfig}

\draft 
\newtheorem{theorem}{Theorem}[section]
\newtheorem{lemma}[theorem]{Lemma}

\theoremstyle{definition}

\newtheorem{proposition}[theorem]{Proposition}

\theoremstyle{remark}

\usepackage{array}
\usepackage{amssymb}
\usepackage{latexsym}
\usepackage{caption}
\usepackage{amsfonts}

\usepackage{placeins}

\theoremstyle{plain} \numberwithin{equation}{section}
\begin{document}


\title{Splints of root systems of Lie Superalgebras} 



\author{Rudra Narayan Padhan}

\email{rudra.padhan6@gmail.com}
\affiliation{Department of Mathematics, National Institute of
Technology Rourkela, Odisha- 769008, India}
\author{K.C. Pati}
\email[Corresponding Author email: ]{kcpati@nitrkl.ac.in}
\affiliation{Department of Mathematics, National Institute of
Technology Rourkela, Odisha- 769008, India
}%
\begin{abstract}
Splints of root systems of simple Lie algebras appears naturally on studies of embedding of reductive subalgebras. A splint can be used to construct branching rules as implementation of this idea simplifies calculation of branching coefficient. We extend the concept of splints to classical Lie superalgebras case as these algebras have wide applications in physics. In this paper we have determined the splints of root system of all classical Lie superalgebras and hope to contribute a small step in the direction of representation of these algebras.
\end{abstract}
\maketitle 
\section{Introduction}

Now a days it is being felt the role of Lie algebras in explaining many physical phenomena is inevitable and desirable. Starting from the application of Lie algebras like $su(2),su(3),so(3)$ in particle physics; to the application of $g_{2}$, $e_{8}$ and their so called extension, affine and hyperbolic Kac-Moody algebras in conformal field theories, string theories, M-theories etc; the role of these types of algebras have increased by leaps and bounds. To be more specific, the structure of these algebras as well as their representation theories play an important role in various branches of physics. For example, in elementary particle physics and especially in model building it is quite important to have effective branching rules for Lie algebra representation. 
\par
A splint \cite{Richter2012} of a root system for a simple Lie algebra appears naturally on studies of (regular) embedding of reductive subalgebras. A splint can be used to construct branching rules. Now it is well understood that implementation of splints properties drastically simplifies calculation of branching coefficient.
\par
An embedding \cite{ Ly2015, L1996} $\iota$ of a root system $\Delta$  in to a root system $\Delta^{\prime}$ is a bijection map of roots $\Delta$ to a (proper) subset of $\Delta^{\prime}$ that commute with vector composition law in $\Delta$ and $\Delta^{\prime}$.
\[\iota:\Delta\longrightarrow \Delta^{\prime}\]
\[\iota(\alpha+\beta)=\iota(\alpha)+\iota(\beta), \forall \alpha,\beta \in \Delta.\]
Note that the image Im($\iota$) must not inherit the root system properties with the exception of addition rules equivalent to the addition rules in $\Delta_{1}$ (for preimages).Two embeddings $\iota_{1}$ and $\iota_{2}$ can be splinter of $\Delta$ when the later can be represented a disjoint union of images  $\iota_{1}$ and $\iota_{2}$. The term splint was introduced by D.Ritcher\cite{Richter2012} where a classification of splints for simple Lie algebras was obtained. At the same time there it was also mentioned that a splint must have tight connections with the injection fan construction. A fan \cite{Nazarov2012} $\Gamma\subset\Delta$ was introduced as a subset of a root system describing recurrence properties of branching coefficient for maximal embeddings. Injection fan is an efficient tool to study branching rules. It is now known that splint is a natural tool to study reduction properties of $g$-modules with respect to a subalgebra $a\hookrightarrow g$. There is a one to one correspondence between weight multiplicities in irreducible modules of splint and branching coefficient for a reduced module.
\par
In many mathematical physics application like supersymmetry we require some algebraic structure which can be readily transcribed for the propose of application to both bosonic and fermionic sectors in a systematic and consistent framework. So it is natural to visualize, evaluate and interpret possible consequences of supersymmetic extension of these type simple Lie algebras. The constructs in this case so generated are called simple Lie superalgebras. 
\par
Like Lie algebra, Lie superalgebras have wide applications in physics and so called classical Lie superalgebras have been classified which have properties similar to simple Lie algebra.  Having this in mind in this paper we construct the splints of Lie superalgebras. Hope this paper will be  a small step forward in this direction of calculation of branching coefficient on different branching rule similar to the role of splints in case of Lie algebras. We also hope the determination of splints for classical Lie superalgebras will pave the way for simplification of calculation of branching coefficients on different branching rule in representation theory. The theory of classical Lie superalgebras \cite{Frappat1989} runs quite parallel to that of Lie algebras, however the migration from Lie algebras to Lie superalgebras is not so direct as perceived. Determination of splints is mainly based on the root system of corresponding algebras.The theorem and techniques used to prove the main theorems in our paper runs similar to that of Ritcher \cite{Richter2012} on splints of Lie algebras. But there are some major differences which we like to mention here. At the outset we point out that in case of  Lie algebra all the root bases of a particular algebra are equivalent. This statement immediately implies that under a transformation of the Weyl group the root system will be transformed into an equivalent one with same Dynkin diagram. However in case of Lie superalgebra, a particular superalgebra may have many inequivalent root systems and hence different Dynkin diagrams due to presence of degenerate and non-degenerate odd roots along with bosonic roots. Out of all these bases the one which contains the least number of odd roots is called a distinguished basis. In this paper we restrict ourselves to the distinguished basis only.

\par
 In case of Lie algebras we have Weyl reflection with respect to one root (even) type only but in case of Lie superalgebras we have Weyl reflection with respect to both even and odd root. However  the Weyl reflections with respect to odd roots do not respect grading; as a result an even root may be mapped to a odd root and vice versa. So in consistent with the given definition of splint, embedding etc. for Lie superalgebras, we consider Weyl reflections only with respect to even roots. Similar attempts \cite{Ransingh2013} have been taken earlier by some authors resulting some partial results for splints but lacking in mathematical rigorousness and proofs. So in this paper we tried to classify the splints of Lie superalgebras once again which has been lying unsolved for many years.

\par
The aim in this article is therefore to provide all instances of this splintering of root systems for classical Lie superalgebras. This is achieved in a case-by-case analysis.
\par
The paper is organized as follows. After a brief introduction to the term splints and motivation for classifying splints of Lie superalgebras in section I we present the root systems of classical Lie superalgebras $A(m,n),B(m,n),B(0,n),C(n+1),D(m,n)$ in sections III, IV and V respectively. Before this in short we give some definitions in section II. We prove lemmas, propositions etc in each chapter corresponding to the individual type of Lie superalgebra which are helpful in determining the splints of the corresponding algebras. At the end of each section we provide a table which lists all the splints obtained through case by case approach. Similar studies are being done in section VI for exceptional Lie superalgebras $G(3),F(4)$ and $D(2,1;\alpha)$. Section VII contains few concluding remarks.

\section{Definitions}
Let $\Delta$ and $\Delta^{\prime}$ be positive root systems of two different Lie superalgebras with  $\Delta=\Delta_{0}+\Delta_{1}$ and $\Delta^{\prime}=\Delta^{\prime}_{0}+\Delta^{\prime}_{1}$  where $\Delta_{0},\Delta_{1}$ and $\Delta^{\prime}_{0},\Delta^{\prime}_{1}$ are even and odd roots of $\Delta$ and $\Delta^{\prime}$ respectively. Then the map  $\iota:\Delta\hookrightarrow \Delta^{\prime}$
is an embedding if 
\begin{enumerate}
\item $\iota$ is a injective function and  $\iota(\gamma)= \iota(\alpha)+\iota(\beta)$ for all $\alpha,\beta,\gamma \in \Delta$ such that $\gamma = \alpha+\beta $
\item $\iota(\Delta_{0})\subseteq \Delta^{\prime}_{0}$ and $\iota(\Delta_{1})\subseteq \Delta^{\prime}_{1}$ .
\end{enumerate}
A root system $\Delta$ splinters as $(\Delta_{1},\Delta_{2})$ if there are two embedding   $\iota_{1}:\Delta_{1}\hookrightarrow \Delta$  and $\iota_{2}:\Delta_{2}\hookrightarrow \Delta$ where
\begin{enumerate}
\item $\Delta$ is the disjoint union of the images of $\iota_{1}$ and $\iota_{2}$ and
\item neither the rank of $\Delta_{1}$ nor the rank of $\Delta_{2}$ exceeds the rank of $\Delta$.
\end{enumerate}

Suppose $\iota:\Delta\hookrightarrow \Delta^{\prime}$ is embedding and suppose that $(,)_{0}$ and $(,)_{1}$ are normalization of $\Delta$ and $\Delta^{\prime}$ respectively. Then the embedding $\iota$ is metric if there is a non-zero integer scalar $\lambda$ such that $(\alpha,\beta)_{0} = \lambda(\iota(\alpha),\iota(\beta))_{1}$ for $\alpha,\beta \in \Delta$ and non-metric otherwise.\\

Here we have found all the splints up to equivalence with Weyl group $W$ (Weyl reflections are with respect to even roots only). If $\Delta$ is a distinguished simple root system then the splints $(\Delta_{1},\Delta_{2})$ and $(\Delta_{1}^\prime,\Delta_{2}^\prime)$ of $\Delta$ are equivalent, if there exists $\sigma \in W$ such that $\sigma. (((\Delta_{1}  \cup (-\Delta_{1}))  |_{\Delta_{0}},(\Delta_{2}  \cup (-\Delta_{2}))  |_{\Delta_{0}}) = ((\Delta_{1}^\prime  \cup (-\Delta_{1}^\prime))  |_{\Delta_{0}},(\Delta_{2}^\prime  \cup (-\Delta_{2}^\prime))  |_{\Delta_{0}})$ and similar restriction for odd roots of $\Delta$ also.
Here we like to mention that Lie superalgebras have Weyl reflections with respect to both isotopic and non-isotopic odd roots. However, in that case we get non-equivalent classes, because grading will not be respected. \\
\section{Splints of Lie superalgebra $A(m-1,n-1)$}
The basic Lie superalgebra $A(m-1,n-1)$ has rank $m+n-1$ and the positive root system is given by 
\[\Delta=\{\varepsilon_{i}-\varepsilon_{j},\delta_{k}-\delta_{l},\delta_{k}-\varepsilon_{i}: 1\leq i\neq j \leq m , 1\leq k \neq l \leq n \},\]
with the normalization
\[(\varepsilon_{i},\varepsilon_{j})=\delta_{ij},~~(\delta_{k},\delta_{l})=\delta_{kl},~~(\varepsilon_{i},\delta_{k})=0 ~~~~~ for ~~1\leq i,j \leq m , 1\leq k,l \leq n .\]

 \begin{lemma}
 If $\Delta$ is a distinguished simple root system and $\Delta\hookrightarrow A(m-1,n-1)$, then $\Delta \cong A(r,s)$ for some $r\leq m-1,~s\leq n-1$.
 \end{lemma}
 \begin{proof}
 As the highest root of $A(m-1,n-1)$ is a linear combination of distinguished simple roots,then every coefficient is equal to 1 .
 \end{proof}
 \begin{lemma}
 $A(m-1,0)$ and $A(0,n-1)$ are metrically embedded in $A(m-1,n-1)$. 
 \end{lemma}
 \begin{lemma}
 If $A(r_{1},s_{1})\hookrightarrow A(m-1,n-1)$ and $A(r_{2},s_{2})\hookrightarrow A(m-1,n-1)$ are embeddings with disjoint images, then $r_{1}+r_{2}\leq m ,~ s_{1}+s_{2}\leq n$ .
 \end{lemma}
 \begin{proof}
 As $A_{l}\hookrightarrow A_{n}$ and $ A_{k}\hookrightarrow  A_{n}$ are embeddings with disjoint images, then $k+l\leq n$.
 \end{proof}
 \begin{lemma}
 Suppose $m\geq3$,~$n\geq3$ and either $r\geq3,~s\geq2$ or $r\geq2,~s\geq3$. If $A(m-1,n-1)$ has a splint where $A(r-1,s-1)$ is a component, then $A(m-2,n-2)$ has a splint having $A(r-2,s-2)$ as a component.
 \end{lemma}
 \begin{proof}
 Suppose $(\Delta_{1},\Delta_{2})$ is a splint of $A(m-1,n-1)$ with  $\iota:A(r-1,s-1)\hookrightarrow \Delta_{1}$ as a component. Without loss of generality, one may assume that the roots in the image of $i$ have the form $\{\varepsilon_{i}\pm\varepsilon_{j},\delta_{k}\pm\delta_{l},\delta_{k}\pm\varepsilon_{i}\}$ where $1\leqslant i\neq j\leqslant r , 1\leqslant k \neq l\leqslant s$. If we are restricting the splint to the embedding $\iota_{1}:A(m-2,n-2)\hookrightarrow A(m-1,n-1)$ and all the components are embedded metrically, this yields a splint of $A(m-2,n-2)$ having $A(r-2,s-2)$ as a component.
 \end{proof}
 \begin{proposition}
Assume $m,n \geq 6$ and if $(\Delta_{1},\Delta_{2})$ is a splint of $A(m-1,n-1)$ having $A(r,s)$ as a component, then $r \in \{0,1,m-1,m-2\}$ and $s \in \{0,1,n-1,n-2\}$
 \end{proposition}
 \begin{proof}
 We can argue by preceding results and table of $A(m-1,n-1)$.
 \end{proof}
 \begin{lemma}
Suppose $\Delta$ is a positive root system of a Lie superalgebra and $\Delta^{\prime}$ is a root system of $A(m-1,n-1)$ or $D(m,n)$. If $\iota:\Delta\hookrightarrow \Delta^{\prime}$ is an embedding, then $\Delta$ is either $A(m-1,n-1)$ or $D(m,n)$ and $\iota$ is metric.
\end{lemma}
\begin{proof}
 $A(m-1,n-1),D(m,n)$ are simply laced, hence $\Delta$ has type  $A(m-1,n-1)$ or $D(m,n)$. Suppose rank of  $\Delta$ is 2, then  $\Delta$ is either $A(0,1),A(1,0)$ or $A(1,1)$. These are metrically embedded in $A(m-1,n-1)$. Suppose rank of  $\Delta$ is greater than 2, then given any odd root $ \alpha \in \Delta$ there is a even root $\beta \in  \Delta$ such that $\alpha+\beta \in  \Delta$. Hence, we always have an embedding $A(1,0)$ or $A(0,1) \hookrightarrow  \Delta$ . As every embedding $A(1,0)$ or $A(0,1) \hookrightarrow  \Delta^{\prime}$ is metric , so $\iota$ is metric.
\end{proof}

\begin{lemma}
 $F(4)$ and $G(3)$ are not embedded in $A(m-1,n-1)$, $B(m,n)$, $C(n+1)$ and $D(m,n)$
 and $D(2,1;\alpha)$ is not embedded in $A(m-1,n-1)$ and $C(n+1)$.
\end{lemma}
\begin{proof}
As even root in $A(m-1,n-1)$, $B(m,n)$, $C(n+1)$ and $D(m,n)$ are linear combination of distinguished simple roots with coefficient one.
\end{proof}
If $(\Delta_{1},\Delta_{2})$ is a splinter of $A(m-1,n-1)$ and $\Delta_{1}\cap A(m-1,0)\neq\phi$ and $\Delta_{2}\cap A(m-1,0)\neq\phi$ , then we find a splinter of $A(m-1,0)$ if we restrict to $\Delta_{1}$ and  $\Delta_{2}$. As $A(0,0)$ has only one odd root, so $A(0,0)$ does not splint. So all the splinter of $A(m,n)$ which are given in the table of $A(m-1,n-1)$ are explicitly described below.
\begin{enumerate}
\item The splinter ($A(2,n)+A_{2},2D_{2}+2nA(0,0)$) of $A(4,n)$ is given by
\begin{align*}
\Delta_{1} &= \{\varepsilon_{i}-\varepsilon_{j},\delta_{k}-\delta_{l},\delta_{k}-\varepsilon_{j}: 1\leq i \neq j \leq 3, 1\leq k \neq l \leq n\},\\
\Delta_{2} &=\{\varepsilon_{1}-\varepsilon_{j},\varepsilon_{2}-\varepsilon_{j},\delta_{k}-\varepsilon_{j}:4\leq j \leq 5, 1\leq k\leq n\}.
\end{align*}

\item The splinter ($A(m-1,0)+A_{n-1},(mn-m)A(0,0)$) of $A(m-1,n-1)$ is given by
 
\begin{align*}
\Delta_{1} &= \{\varepsilon_{i}-\varepsilon_{j},\delta_{1}-\varepsilon_{j},\delta_{k}-\delta_{l}: 1\leq i \neq j \leq m, 1\leq k \neq l \leq n\},\\
\Delta_{2} &= \{\delta_{k}-\varepsilon_{l}:2\leq k \leq n, 1\leq l\leq m\}.
\end{align*}
\item The splinter ($A(0,n-1)+A_{m-1},(mn-n)A(0,0)$) of $A(m-1,n-1)$ is given by

\begin{align*}
 \Delta_{1}&= \{\delta_{i}-\delta_{j},\delta_{j}-\varepsilon_{1}, \varepsilon_{k}-\varepsilon_{l} :1\leq i \neq j \leq n,1\leq k \neq l \leq m\},\\
\Delta_{2}&=\{\delta_{k}-\varepsilon_{l}: 1\leq k \leq m, 2\leq l\leq n\}.
\end{align*}
\item If $m-n=1$ ,then the splinter ($A(n-1,n-1)$ , $nA_{1}+nA(0,0)$) of $A(m-1,n-1)$ is given by
 \begin{align*}
 \Delta_{1}&= \{\varepsilon_{i}-\varepsilon_{j},\delta_{i}-\delta_{j},\delta_{i}-\varepsilon_{j}:1\leq i\neq j\leq n\},\\
  \Delta_{2}&= \{\varepsilon_{i}-\varepsilon_{m},\delta_{i}-\varepsilon_{m}: 1 \leq i \leq n \}.
  \end{align*}
\item The splinter   ($A(1,n)+A_{m-2}$ , $(m-2)A_{1}+n(m-2)A(0,0)$) of $A(m-1,n-1)$ is given by
\begin{align*} 
\Delta_{1}&= \{\varepsilon_{1}-\varepsilon_{2},\delta_{k}-\delta_{l},\delta_{k}-\varepsilon_{1},\delta_{k}-\varepsilon_{2}: 1\leq k \neq l \leq n\} \cup \{\varepsilon_{i}-\varepsilon_{j}:2\leq i \neq j \leq m\},\\
\Delta_{2}&= \{\varepsilon_{1}-\varepsilon_{j},\delta_{i}-\varepsilon_{j}: 3\leq j \leq m ,1\leq i \leq n\}.
\end{align*}
\item The splinter   ($A(m,1)+A_{n-2}$ , $(n-2)A_{1}+m(n-2)A(0,0)$) of $A(m-1,n-1)$ is given by
\begin{align*}
\Delta_{1}&= \{\varepsilon_{k}-\varepsilon_{l},\delta_{1}-\delta_{2},\delta_{1}-\varepsilon_{k},\delta_{2}-\varepsilon_{k}: 1\leq k \neq l \leq m\} \cup \{\delta_{i}-\delta_{j}:2\leq i \neq j \leq n\},\\
\Delta_{2}&= \{\delta_{1}-\delta_{j},\delta_{j}-\varepsilon_{k}: 3\leq j \leq n ,1\leq k \leq m\}.
\end{align*}
\item For $m,n\geq 2$, the splinter  ($A(m-2,n-2)+A(1,0), (m+n-3)A_{1}+(m+n-3)A(0,0)$) of $A(m-1,n-1)$ is given by
\begin{align*}
 \Delta_{1}&= \{\varepsilon_{i}-\varepsilon_{j},\delta_{k}-\delta_{l},\delta_{k}-\varepsilon_{i}: 2\leq i \neq j \leq m,2\leq k \neq l \leq n\} \cup \{\varepsilon_{1}-\varepsilon_{2},\delta_{1}-\varepsilon_{1},\delta_{1}-\varepsilon_{2}\},\\
 \Delta_{2}&= \{\varepsilon_{1}-\varepsilon_{i},\delta_{1}-\delta_{k},\delta_{1}-\varepsilon_{i},\delta_{k}-\varepsilon_{1}: 3\leq i \neq m ,2\leq k \leq n\}.
 \end{align*}
 \item For $m,n\geq 2$, the splinter ($A(m-1,n-2),(n-1)A_{1}+mA(0,0)$) of $A(m-1,n-1)$ is given by 
 \begin{align*}
 \Delta_{1}&= \{\varepsilon_{i}-\varepsilon_{j},\delta_{k}-\delta_{l},\delta_{k}-\varepsilon_{i}: 1\leq i \neq j \leq m, 1\leq k \neq l \leq n-1\},\\
 \Delta_{2}&= \{\delta_{k}-\delta_{n},\delta_{n}-\varepsilon_{i}: 1\leq k \leq n-1, 1\leq i \leq m\}.
 \end{align*}
 \item For $m,n\geq 2$, the splinter ($A(m-2,n-1),(m-1)A_{1}+nA(0,0)$) of $A(m-1,n-1)$ is given by 
 \begin{align*}
 \Delta_{1}&= \{\varepsilon_{i}-\varepsilon_{j},\delta_{k}-\delta_{l},\delta_{k}-\varepsilon_{i}: 1\leq i \neq j \leq m-1, 1\leq k \neq l \leq n\},\\
 \Delta_{2}&= \{\varepsilon_{j}-\varepsilon_{m},\delta_{i}-\varepsilon_{m}: 1\leq j \leq m-1, 1\leq i \leq n\}.
 \end{align*}
 \item For $m=n$, the splinter ($A(m-1,m-2),(m-1)A_{1}+mA(0,0)$) of $A(m-1,n-1)$ is given by
  \begin{align*}
  \Delta_{1}&= \{\varepsilon_{i}-\varepsilon_{j},\delta_{k}-\delta_{l},\delta_{k}-\varepsilon_{i}: 1\leq i \neq j \leq m, 1\leq k \neq l \leq m-1\},\\
  \Delta_{2}&= \{\delta_{i}-\delta_{m},\delta_{m}-\varepsilon_{j}: 1\leq i \leq m-1, 1\leq j \leq m\}.
  \end{align*}
  \item  For $m=n$ , the splinter ($A(m-2,m-1),(m-1)A_{1}+mA(0,0)$) of $A(m-1,n-1)$ is given by
 \begin{align*}
  \Delta_{1}&= \{\varepsilon_{i}-\varepsilon_{j},\delta_{k}-\delta_{l},\delta_{k}-\varepsilon_{i}: 1\leq i \neq j \leq m-1, 1\leq k \neq l \leq m\},\\
  \Delta_{2}&= \{\varepsilon_{i}-\varepsilon_{m},\delta_{m}-\varepsilon_{j}: 1\leq i \leq m-1, 1\leq j \leq m\}.
  \end{align*}
  \item For $m=n$, the splinter ($A(m-2,m-2),2(m-1)A_{1}+(2m-1)A(0,0)$) of $A(m-1,n-1)$ is given by
 \begin{align*}
  \Delta_{1}&= \{\varepsilon_{i}-\varepsilon_{j},\delta_{i}-\delta_{j},\delta_{i}-\varepsilon_{j}: 1\leq i \neq j \leq m-1\},\\
  \Delta_{2}&= \{\varepsilon_{i}-\varepsilon_{m},\delta_{i}-\delta_{m},\delta_{i}-\varepsilon_{m},\delta_{m}-\varepsilon_{j}: 1\leq i \leq m-1, 1\leq j \leq m\}.
  \end{align*}
\end{enumerate}

\FloatBarrier
 \begin{table}[b]
\caption {$A(m,n)$} \label{tab:title} 

\begin{tabular}{ | b{3cm} | m{6cm}| m{6cm} | } 
 \hline 
 
   $\Delta$     & $\Delta_{1}$ & $\Delta_{2}$ \\ \cline{1-3}  
      $A(1,0)$ & $A_{1}$ & $2A(0,0)$ \\ 
               
  \hline
   $A(0,1)$ & $A_{1}$ & $2A(0,0)$ \\ 
              
  \hline
   $A(1,1)$ & $A(0,1)$ & $A_{1}+2A(0,0)$ \\ 
               & $A(1,0)$ & $A_{1}+2A(0,0)$ \\
               & $2A_{1}$ & $4A(0,0)$ \\
  
 \hline           
 $A(1,2)$ & $A(0,2)$ & $2A_{1}+3A(0,0)$ \\ 
               & $A(1,1)$ & $2A_{1}+2A(0,0)$ \\
              
               & $A_{2}+2A(0,0)$ & $A(1,0)+2A(0,0)$ \\
                & $A_{1}+A_{2}$ & $6A(0,0)$ \\
  
 \hline                        
 $A(2,2)$ & $A(2,1)$ & $2A_{1}+3A(0,0)$ \\ 
               & $A(1,1)+A_{1}$ & $A(1,0)+2A_{1}+3A(0,0)$ \\
               & $A_{2}+A_{2}$ & $9A(0,0)$ \\
 \hline                  
  $A(0,2)$ & $A_{1}+A(0,1)$ & $A_{1}+A(0,0)$ \\ 
               
               & $A_{2}$ & $3A(0,0)$ \\ 
          
  \hline
  $A(4,4)$ & $A(2,4)+A_{2}$ & $2D_{2}+10A(0,0)$ \\
           & $A(4,2)+A_{2}$ & $2D_{2}+10A(0,0)$ \\
   \hline
     $A(4,n)$ & $A(2,n)+A_{2}$ & $2D_{2}+2nA(0,0)$ \\ 
\hline   
 $A(m-1,n-1)$ &$A(m-1,0)+A_{n-1}$ & $(mn-m)A(0,0)$\\
              &$A(0,n-1)+A_{m-1}$ & $(mn-n)A(0,0)$\\
 \hline

   $A(m-1,n-1)$  &$A(1,n)+A_{m-2}$ & $(m-2)A_{1}+n(m-2)A(0,0)$\\
             &$A(m,1)+A_{n-2}$ & $(n-2)A_{1}+m(n-2)A(0,0)$\\
 \hline 
 if $m-n=1$,  $A(m-1,n-1)$  &$A(n-1,n-1)$ & $nA_{1}+nA(0,0)$\\ 
 \hline   
    \end{tabular} 
\end{table}
\FloatBarrier
 \begin{table}[h]

\begin{tabular}{ | b{3cm} | m{6cm}| m{6cm} | }     
          
 \hline                      
   $A(m-1,m-1)$  &$A(m-1,m-2)$ & $(m-1)A_{1}+mA(0,0)$\\
              &$A(m-2,m-1)$ & $(m-1)A_{1}+mA(0,0)$\\
          &$A(m-2,m-2)$ & $2(m-1)A_{1}+(2m-1)A(0,0)$\\
         
\hline

  $A(m-1,n-1)$  for $m,n\geqslant2$ & $A(m-2,n-2)+A(1,0)$ & $(m+n-3)A_{1}+(m+n-3)A(0,0)$ \\ 
  & $A(m-2,n-1)$ & $(m-1)A_{1}+nA(0,0)$ \\
              & $A(m-1,n-2)$ & $(n-1)A_{1}+ mA(0,0)$  \\
  \hline                  
\end{tabular} 
\end{table} 

\section{Splints of Lie superalgebra $B(m,n)$ and $B(0,n)$ }
The basic Lie superalgebra $B(m,n)$ has rank $m+n$ and the positive root system is given by 
\[\Delta=\{\varepsilon_{i}\pm\varepsilon_{j},\varepsilon_{i},\delta_{k}\pm\delta_{l},2\delta_{k},\delta_{k}\pm\varepsilon_{i},\delta_{k}: 1\leq i\neq j \leq m , 1\leq k \neq l \leq n \},\]
with the normalization
\[(\varepsilon_{i},\varepsilon_{j})=-\delta_{ij},~~(\delta_{k},\delta_{l})=\delta_{kl},~~(\varepsilon_{i},\delta_{k})=0 ,~~~~~ for ~~1\leq i,j \leq m , 1\leq k,l \leq n .\]
\begin{lemma}
$C(n+1)$ is not embedded in $B(m,n)$ for $m>3,~n\geq2$.
\end{lemma}
\begin{proof}
Suppose $C(n+1) \hookrightarrow B(m,n)$. As the even roots of $C_{3}$ does not embed in $B_{m}$ for $m\geq2$, hence the image of even part of $C(n+1)$ under the map $\iota$ is $\{\delta_{k}\pm\delta_{l},2\delta_{k}\}$ where $1\leq k \neq l \leq n$. Now without loss of generality, the distinguished simple root system of $C(n+1)$ under the map $\iota$ is $\{\alpha_{1},\alpha_{2},\cdots\alpha_{n-1},2\alpha_{n}+2\alpha_{n+1}+\cdots+2\alpha_{n-1}\} \cup \{\beta\}$, where $\beta$ is an odd root of $B(m,n)$ and $\alpha_{1}= \delta_{1}-\delta_{2}$, $\alpha_{2}= \delta_{2}-\delta_{3}$,$\cdots$, $\alpha_{n-1}= \delta_{n-1}-\delta_{n}$ which belong to distinguished simple roots of even part of $B(m,n)$. $C(n+1)$ has an odd root $\beta+\alpha_{1}+\alpha_{2}+\cdots+\alpha_{n-1}+2\alpha_{n}+2\alpha_{n+1}+\cdots+2\alpha_{n-1}$ but $B(m,n)$ has no such odd root.
\end{proof}
\begin{lemma}
$B(m,n-1)$ and $B(m-1,n)$ are not a component of $B(m,n)$.
\end{lemma}
\begin{proof}
Suppose $(\Delta_{1},\Delta_{2})$ is a splint of $B(m,n)$ and $B(m,n-1)\hookrightarrow \Delta_{1}$. Without loss of generality, we may assume that the image of $B(m,n-1)$ under $\iota$ is $\{\varepsilon_{i}\pm\varepsilon_{j},\varepsilon_{i},\delta_{k}\pm\delta_{l},2\delta_{k},\delta_{k}\pm\varepsilon_{i},\delta_{k}:1\leqslant i\neq j\leqslant m , 1\leqslant k \neq l\leqslant n-1\}$. $B(0,n)$ is metrically embedded in $B(m,n)$. So restricting the splints to $B(0,n)$, we get a splints of $B(0,n)$ say $\Delta_{2}^\prime=\{\delta_{k}\pm\delta_{n},2\delta_{k} ,\delta_{k}:1\leq k \leq n-1 \}$.  But the rank of $\Delta_{2}^\prime$ is greater then $B(0,n)$. Hence we get a contradiction.
\end{proof}
We can describe all the splinter of $B(m,n)$ in the following way,
\begin{enumerate}
\item The splinter $(A(0,1),A_{1}+2A(0,0))$ of $B(1,1)$ is given by
\begin{align*}
  \Delta_{1}&=\{\delta_{1}-\varepsilon_{1},\varepsilon_{1},\delta_{1}\},\\
   \Delta_{2}&=\{\delta_{1}+\varepsilon_{1},2\delta_{1}\}.
   \end{align*}
   \item The splinter  $(A(0,1)+A_{1},3A_{1}+4A(0,0))$ of $B(1,2)$ is given by
  \begin{align*}
   \Delta_{1}&=\{\delta_{1}-\delta_{2},\delta_{1}-\varepsilon_{1},\delta_{2}-\varepsilon_{1},\varepsilon_{1}\},\\
   \Delta_{2}&=\{\delta_{1}+\delta_{2},2\delta_{1},2\delta_{2},\delta_{2}+\varepsilon_{1},\delta_{1}+\varepsilon_{1},\delta_{1},\delta_{2}\}.
   \end{align*}
   Another splinter of $B(1,2)$ is  ($B(0,2),A_{1}+4A(0,0)$) which is given by
 \begin{align*}
   \Delta_{1}&=\{\delta_{1}\pm\delta_{2},2\delta_{1},2\delta_{2},\delta_{1},\delta_{2}\},\\
   \Delta_{2}&=\{\varepsilon_{1},\delta_{1}\pm\varepsilon_{1},\delta_{2}\pm\varepsilon_{1}\}.
   \end{align*}
\item The splinter ($A(1,1)+2A_{1},4A_{1}+6A(0,0)$) of $B(2,2)$ is given by
   \begin{align*}
   \Delta_{1}&=\{\varepsilon_{1}-\varepsilon_{2},\delta_{1}-\delta_{2},\delta_{1}-\varepsilon_{1},\delta_{1}-\varepsilon_{2},\delta_{2}-\varepsilon_{1},\delta_{2}-\varepsilon_{2}\}\cup \{\varepsilon_{1}+\varepsilon_{2},\delta_{1}+\delta_{2}\},\\
  \Delta_{2}&=\{\varepsilon_{1},\varepsilon_{2},2\delta_{1},2\delta_{2}\}\cup\{\delta_{1}+\varepsilon_{1},\delta_{1}+\varepsilon_{2},\delta_{2}+\varepsilon_{1},\delta_{2}+\varepsilon_{2},\delta_{1},\delta_{2}\}.
  \end{align*}
\item $B(0,2)$ has two additional splints. The first one ($A_{2},A_{1}+2A(0,0)$) is given by 
\begin{align*}
 \Delta_{1}&= \{\delta_{1}\pm\delta_{2},2\delta_{2}\},\\ 
 \Delta_{2}&= \{2\delta_{1},\delta_{1},\delta_{2}\}.
 \end{align*}
 and the second one ($2A_{1}+A(0,0),2A_{1}+A(0,0)$) is given by
  \begin{align*}
  \Delta_{1}&= \{\delta_{1}+\delta_{2},2\delta_{2}\}\cup\{\delta_{1}\}, \\
 \Delta_{2}&= \{\delta_{1}-\delta_{2},2\delta_{1}\}\cup\{\delta_{2}\}.
 \end{align*}
 \item The splinter  ($A_{1}+B(0,2),A_{1}+A_{2}+A(0,0)$) of $B(0,3)$ is given by \begin{align*}
  \Delta_{1}&= \{\delta_{1}\pm\delta_{2},2\delta_{1},2\delta_{2},\delta_{1},\delta_{2}\}\cup\{\delta_{2}-\delta_{3}\},\\
 \Delta_{2}&=\{\delta_{2}+\delta_{3}\}\cup\{\delta_{1}\pm\delta_{3},2\delta_{3}\}\cup\{\delta_{3}\}.
 \end{align*}
             
\item The splinter   ($D_{n},nB(0,1))$ of $B(0,n)$ is given by
\begin{align*}
 \Delta_{1}&= \{\delta_{k}\pm\delta_{l}:1\leq k \neq l \leq n\},\\
 \Delta_{2}&= \{2\delta_{k},\delta_{k}:1\leq k \leq n\}.
 \end{align*}
 \item The splinter   ($C_{n},nA(0,0))$ of $B(0,n)$ is given by
\begin{align*}
 \Delta_{1}&= \{\delta_{k}\pm\delta_{l},2\delta_{k}:1\leq k \neq l \leq n\},\\
 \Delta_{2}&= \{\delta_{k}:1\leq k \leq n\}.
 \end{align*}
\item The splinter $(B(0,n)+B_{m} , 2mnA(0,0))$ and $(B_{m}+C_{n},  (2mn+n)A(0,0))$ of $B(m,n)$ are equivalent because when we restrict to even roots both the splinter are same. Hence consider the splinter $(B(0,n)+B_{m} , 2mnA(0,0))$ which is given by
\begin{align*} 
 \Delta_{1}&= \{\delta_{k}\pm\delta_{l},2\delta_{k},\delta_{k}\}\cup\{\varepsilon_{i}\pm\varepsilon_{j}\}, ~where ~ 1\leq i \neq j \leq m, 1\leq k \neq l \leq n\\ 
 \Delta_{2}&= \{\delta_{i}-\varepsilon_{k}: 1\leq i \neq j \leq m, 1\leq k \neq l \leq n\}.
\end{align*}
 \item  For $m\geqslant 2,~n\geqslant 1$ the splinter $(D(m,n), mA_{1}+nA(0,0))$ is given by
 \begin{align*}
  \Delta_{1}&= \{\delta_{k}\pm\delta_{l},2\delta_{k},\varepsilon_{i}\pm\varepsilon_{j},\delta_{k}-\varepsilon_{i}: 1\leq i \neq j \leq m, 1\leq k \neq l \leq n\},\\ 
  \Delta_{2}&= \{\delta_{i},\varepsilon_{k}: 1\leq i \leq m, 1\leq k  \leq n\}.
  \end{align*}
\end{enumerate}

\begin {table}[h]
\caption {$B(m,n)$} \label{tab:title} 

\begin{tabular}{ | b{3cm} | m{6cm}| m{6cm} | } 
\hline 
 $\Delta$     & $\Delta_{1}$ & $\Delta_{2}$ \\ \cline{1-3} 
  $B(0,1)$ & $A_{1}$ & $A(0,0)$ \\
  \hline
  $B(1,1)$ & $A(0,1)$ & $A_{1}+2A(0,0)$ \\
  \hline
   $B(1,2)$ & $B(0,2)$ & $A_{1}+4A(0,0)$ \\
            & $A(0,1)+A_{1}$ & $3A_{1}+4A(0,0)$ \\
   \hline
    $B(2,1)$ &$A(1,0)+A_{1}$ & $3A_{1}+3A(0,0)$ \\
    \hline
   
    $B(2,2)$ & $A(1,1)+2A_{1}$ & $4A_{1}+6A(0,0)$ \\
    \hline
    \end{tabular} 
\end {table}
    
    \begin {table}[h]
\begin{tabular}{ | b{3cm} | m{6cm}| m{6cm} | } 
     \hline
    $B(0,2)$ & $A_{2}$ & $A_{1}+2A(0,0)$ \\
             & $2A_{1}+A(0,0)$ & $2A_{1}+A(0,0)$ \\
     \hline
    $B(0,3)$ & $A_{1}+B(0,2)$ & $A_{1}+A_{2}+A(0,0)$ \\
    \hline
    $B(0,n)$ & $D_{n}$ & $nB(0,1)$ \\
     \hline
    $B(0,n)$ & $C_{n}$ & $nA(0,0)$ \\
    \hline
    $B(m,n)$ & $B(0,n)+B_{m}$ & $2mnA(0,0)$ \\
             
    \hline
    $B(m,n)$ for $m\geqslant 2, n\geqslant 1$& $D(m,n)$ & $mA_{1}+nA(0,0)$ \\
    \hline
\end{tabular} 
\end {table} 
\section{Splints of Lie superalgebra $C(n+1)$}
The basic Lie superalgebra $C(n+1)$ has rank $n+1$ and the positive root system is given by 
\[\Delta=\{\delta_{k}\pm\delta_{l},2\delta_{k},\varepsilon\pm\delta_{k}: 1\leq k \neq l \leq n \},\]
with the normalization
\[(\varepsilon,\varepsilon)=1,~~(\delta_{k},\delta_{l})=-\delta_{kl},~~(\varepsilon,\delta_{k})=0 ~~~~~ for ~~1\leq k,l \leq n .\]
\begin{lemma}
$C(n)$ is not a component of $C(n+1)$ for $n\geqslant3$.
\end{lemma}  
\begin{proof}
Suppose $(\Delta_{1},\Delta_{2})$ is a splint of $C(n+1)$ and $C(n)\hookrightarrow \Delta_{1}$. Then the other components of $\Delta_{1}$ are isomorphic to either $A_{1}$ or $A(0,0)$ and components of $\Delta_{2}$ are isomorphic to either $A_{1}$ or $D_{2}$, which implies rank of $\Delta_{2}$ is greater than $n+1$ . Hence a contradiction.
\end{proof}    
\begin{lemma}
$B(m,n)$ is not embedded in $C(n+1)$.
\end{lemma}
\begin{proof}
Consider an even root  $\alpha$ and the odd root $\beta$ in the distinguished simple root system of $B(m,n)$. Then $\alpha+2\beta$ is a odd root, but $C(n+1)$ has no such roots.
\end{proof}
\begin{lemma}
$B(0,n)$ is not embedded in $C(n+1)$.
\end{lemma}
\begin{proof}
$B(0,n)$ has an odd root $\beta$ such that $2\beta$ is an even root, but $C(n+1)$ has no such roots.
\end{proof}
\begin{lemma}
$D(r,s)$ is not embedded in $C(n+1)$ for $r,s\leq n$
\end{lemma}
\begin{proof}
We can observe from the properties of $D(r,s)$ that there are even roots which are linear combination of even as well as odd roots. But $C(n+1)$ does not have such type of even roots.
\end{proof}

We can describe all the splinters of $C(n+1)$ in the following way,
\begin{enumerate}
\item For $n\geq1$, $C(n+1)$ has a splint that is ($C_{n},2nA(0,0)$) is given by  
\begin{align*}
\Delta_{1}&=\{\delta_{k}\pm\delta_{l},2\delta_{k}:1\leq k \neq l \leq n\},\\
\Delta_{2}&=\{\varepsilon\pm\delta_{k}:1\leq k \neq l \leq n\}.
  \end{align*}   
\item $C(3)$ has two additional splinters that are ($A_{2}, C(2)+2A(0,0)$) which is given by 
 
  \begin{align*}
  \Delta_{1}&=\{\delta_{1}\pm\delta_{2},2\delta_{1}\},\\
  \Delta_{2}&=\{2\delta_{2}\varepsilon\pm\delta_{2}\}\cup\{\varepsilon\pm\delta_{1}\}
   \end{align*}
  and ($C(2)+A_{1},C(2)+A_{1}$) which is given by  
   \begin{align*}
   \Delta_{1}&=\{2\delta_{1},\varepsilon\pm\delta_{1},\}\cup\{\delta_{1}-\delta_{2}\},\\     
  \Delta_{2}&=\{2\delta_{2},\varepsilon\pm\delta_{2},\}\cup\{\delta_{1}+\delta_{2}\}.
   \end{align*}
  \item $C(4)$ has two additonal splinters that are  ($A_{1}+B_{2}+2A(0,0),A_{1}+A_{2}+4A(0,0)$) which is given by
   
  \begin{align*}
  \Delta_{1}&=\{\delta_{2}-\delta_{3}\}\cup\{2\delta_{1},2\delta_{2},\delta_{1}\pm\delta_{2}\}\cup\{\varepsilon\pm\delta_{3}\},\\
  \Delta_{2}&=\{\delta_{2}+\delta_{3}\}\cup\{2\delta_{3},\delta_{1}\pm\delta_{3}\}\cup\{\varepsilon\pm\delta_{1},\varepsilon\pm\delta_{2}\}
   \end{align*}
 and ($C(3)+A_{1},D_{2}+D_{2}+2A(0,0)$) is given by
   \begin{align*}
\Delta_{1}&=\{\delta_{1}\pm\delta_{2},2\delta_{1},2\delta_{2},\varepsilon\pm\delta_{1},\varepsilon\pm\delta_{2}\}\cup\{2\delta_{3}\},\\
 \Delta_{2}&=\{\delta_{1}\pm\delta_{3}\}\cup\{\delta_{2}\pm\delta_{3}\}\cup\{\varepsilon\pm\delta_{3}\}.
 \end{align*}
\end{enumerate}

\begin {table}[h]
\caption {$C(n+1)$} \label{tab:title} 
\begin{center}

\begin{tabular}{ | b{3cm} | m{6cm}| m{6cm} | } 
 \hline 
 $\Delta$     & $\Delta_{1}$ & $\Delta_{2}$ \\ \cline{1-3} 
  $C(2)$ & $A(1)$ & $2A(0,0)$ \\
\hline  
  $C(3)$ & $A_{2}$ & $C(2)+2A(0,0)$\\
         & $C(2)+A_{1}$ & $C(2)+A_{1}$\\

\hline
$C(4)$ & $A_{1}+B_{2}+2A(0,0)$ & $A_{1}+A_{2}+4A(0,0)$\\
       & $C(3)+A_{1}$ & $D_{2}+D_{2}+2A(0,0)$\\
\hline
   $C(n)$ & $C_{n}$ & $2nA(0,0)$  \\
 \hline      
 
\end{tabular} 
\end{center}
\end {table} 
\newpage 
\section{Splints of Lie superalgebra $D(m,n)$}
The basic Lie superalgebra $D(m,n)$ has rank $m+n$ and the positive root system is given by 
\[\Delta=\{\varepsilon_{i}\pm\varepsilon_{j},\delta_{k}\pm\delta_{l},2\delta_{k},\delta_{k}\pm\varepsilon_{i}: 1\leq i\neq j \leq m , 1\leq k \neq l \leq n \},\]
with the normalization
\[(\varepsilon_{i},\varepsilon_{j})=-\delta_{ij},~~(\delta_{k},\delta_{l})=\delta_{kl},~~(\varepsilon_{i},\delta_{k})=0 ~~~~~ for ~~1\leq i,j \leq m , 1\leq k,l \leq n .\]
\begin{lemma}
Suppose $m\geqslant2$,~$n\geqslant3$ and $r\geqslant2$. If $D(m,n)$ has a splinter where $D(m,r)$ is a component, then $D(m,n-1)$ has a splinter having $D(m,r-1)$ as a component.
\end{lemma}

\begin{proof}
Suppose $(\Delta_{1},\Delta_{2})$ is a splint of $D(m,n)$ having $\iota:D(m,r)\hookrightarrow \Delta_{1}$ as a component. Without loss of generality, one can assume that the roots in the image of $\iota$ have the form $\{\varepsilon_{i}\pm\varepsilon_{j},\delta_{k}\pm\delta_{l},2\delta_{k},\delta_{k}\pm\varepsilon_{i}\}$ where $1\leqslant i\neq j\leqslant m , t+1\leqslant k \neq l\leqslant n$ ,when $ r=n-t$ for some $t\in \mathbb{Z}$. 
    Consider restricting the splint to the embedding $\iota_{1}:D(m,n-1)\hookrightarrow D(m,n)$, where the image of $\iota_{1}$ consists of roots of the form $\{\varepsilon_{i}\pm\varepsilon_{j},\delta_{k}\pm\delta_{l},2\delta_{k},\delta_{k}\pm\varepsilon_{i}\}$ where $1\leqslant i\neq j\leqslant m ,2\leqslant k\neq l\leqslant n$. Since $\iota_{1}$ is metric and all components of $\Delta_{1}$ and $\Delta_{2}$ are embedded metrically, this yields a splint of $D(m,n-1)$ having $D(m,r-1)$ as a component.
\end{proof}
\begin{proposition}
For $m\geqslant 4$, $n\geqslant 1$ or $m\geqslant2$, $n>3$ and either $n-2\leqslant m$, $m\geqslant n$ or $m-2\leqslant n , n \geqslant m$. If  $(\Delta_{1},\Delta_{2})$ is a splinter of $D(m,n)$ having $D(r,s)$ as a component, then $r,s\in\{1,n-1,m-1\}$.
\end{proposition}
\begin{proof}
One may argue by contradiction using previous lemma and table of $D(m,n)$.
\end{proof}
\begin{lemma}
$A(m-1,n-1)$ is not a component of $D(m,n)$ for either $m\geqslant 4,~n\geqslant 1$ or $m\geqslant 2 ,~n>3$ and not a component of $C(n+1)$ for $n\geq4$.
\end{lemma}
\begin{proof}
Suppose $(\Delta_{1},\Delta_{2})$ is a splint of $D(m,n)$ and $A(m-1,n-1)\hookrightarrow \Delta_{1}$. Then other components of $\Delta_{1}$ are isomorphic to $A_{1}$ or $A(0,0)$. Also all components of $\Delta_{2}$ are  isomorphic to $A_{1}$ or $A(0,0)$. Then rank of $\Delta_{2}$ is greater than $m+n$, which is a contradiction. Similar argument for $C(n+1)$ .
\end{proof}
\begin{lemma}
$B(m,n)$and $B(0,n)$ are not embedded in $D(m,n)$.
\end{lemma}
\begin{proof}
As $B(m,n)$ has a $\beta$ odd root such that $2\beta$ is an even root, but $D(m,n)$ has no such root.
\end{proof}
\begin{lemma}
$C(n+1)$ is not embedded in $D(m,n)$.
\end{lemma}
\begin{proof}
As $D(m,n)$ is embedded in $B(m,n)$ and $C(n+1)$ is not embedded in $B(m,n)$. 
\end{proof}

We can describe all the splinters of $D(m,n)$ in the following way,
\begin{enumerate}
\item The splinter ($A(1,0),2A_{1}+2A(0,0)$) of  $D(2,1)$ is given by
 \begin{align*}
 \Delta_{1}&=\{\varepsilon_{1}-\varepsilon_{2},\delta_{1}-\varepsilon_{1},\delta_{1}-\varepsilon_{2}\},\\
\Delta_{2}&=\{2\delta_{1},\varepsilon_{1}+\varepsilon_{2}\}\cup\{\delta_{1}+\varepsilon_{1},\delta_{1}+\varepsilon_{2}\}.
 \end{align*}
\item The splinter ($A(1,1),4A_{1}+A(0,0)$) of $D(2,2)$ is given by
 \begin{align*}
 \Delta_{1}&=\{\varepsilon_{1}-\varepsilon_{2},\delta_{1}-\delta_{2},\delta_{1}-\varepsilon_{1},\delta_{1}-\varepsilon_{2},\delta_{2}-\varepsilon_{1},\delta_{2}-\varepsilon_{2}\},\\
\Delta_{2}&=\{\varepsilon_{1}+\varepsilon_{2},\delta_{1}+\delta_{2},2\delta_{1},2\delta_{2}\}\cup\{\delta_{1}+\varepsilon_{1},\delta_{1}+\varepsilon_{2},\delta_{2}+\varepsilon_{1},\delta_{2}+\varepsilon_{2}\}.
 \end{align*}
\item The splinter ($A(2,0),4A_{1}+3A(0,0)$) of $D(3,1)$ is given by
 \begin{align*}
 \Delta_{1}&=\{\varepsilon_{1}-\varepsilon_{2},\varepsilon_{1}-\varepsilon_{3},\varepsilon_{2}-\varepsilon_{3},\delta_{1}-\varepsilon_{1},\delta_{1}-\varepsilon_{2},\delta_{1}-\varepsilon_{3}\},\\
\Delta_{2}&=\{\varepsilon_{1}+\varepsilon_{2},\varepsilon_{1}+\varepsilon_{3},\varepsilon_{2}+\varepsilon_{3},2\delta_{1}\}\cup\{\delta_{1}+\varepsilon_{1},\delta_{1}+\varepsilon_{2},\delta_{1}+\varepsilon_{3}\}.
 \end{align*}
\item The splinter($A(2,1)+A_{1},5A_{1}+6A(0,0)$)of $D(3,2)$ is given by 
 \begin{align*}
 \Delta_{1}&=\{\varepsilon_{1}-\varepsilon_{2},\varepsilon_{1}-\varepsilon_{3},\varepsilon_{2}-\varepsilon_{3},\delta_{1}-\delta_{2},\delta_{1}-\varepsilon_{1},\delta_{1}-\varepsilon_{2},\delta_{1}-\varepsilon_{3},\delta_{2}-\varepsilon_{1},\delta_{2}-\varepsilon_{2},\delta_{2}-\varepsilon_{3}\}\cup\{2\delta_{2}\},\\
\Delta_{2}&=\{\varepsilon_{1}+\varepsilon_{2},\varepsilon_{1}+\varepsilon_{3},\varepsilon_{2}+\varepsilon_{3},\delta_{1}+\delta_{2},2\delta_{1}\}\cup\{\delta_{1}+\varepsilon_{1},\delta_{1}+\varepsilon_{2},\delta_{1}+\varepsilon_{3},\delta_{2}+\varepsilon_{1},\delta_{2}+\varepsilon_{2},\delta_{2}+\varepsilon_{3}\}.
 \end{align*}
\item The splinter($A(1,2)+2A_{1},5A_{1}+6A(0,0)$)of $D(2,3)$ is given by 
 \begin{align*}
 \Delta_{1}&=\{\varepsilon_{1}-\varepsilon_{2},\delta_{1}-\delta_{2},\delta_{1}-\delta_{3},\delta_{2}-\delta_{3},\delta_{1}-\varepsilon_{1},\delta_{2}-\varepsilon_{1},\delta_{3}-\varepsilon_{1},\delta_{1}-\varepsilon_{2},\delta_{2}-\varepsilon_{2},\delta_{3}-\varepsilon_{2}\}\cup\{2\delta_{1},2\delta_{2}\},\\
\Delta_{1}&=\{\varepsilon_{1}+\varepsilon_{2},\delta_{1}+\delta_{2},\delta_{1}+\delta_{3},\delta_{2}+\delta_{3},\delta_{1}+\varepsilon_{1},\delta_{2}+\varepsilon_{1},\delta_{3}+\varepsilon_{1},\delta_{1}+\varepsilon_{2},\delta_{2}+\varepsilon_{2},\delta_{3}+\varepsilon_{2}\}\cup\{2\delta_{3}\}.
 \end{align*}
\item For either $n-2\leqslant m$ or $m\geqslant n$ the splinter ($D(m,n-1)+A_{1}, (2n-2)A_{1}+2mA(0,0)$) of $D(m,n)$ is given by
 \begin{align*}
 \Delta_{1}&=\{\varepsilon_{i}\pm\varepsilon_{j},\delta_{k}\pm\delta_{l},2\delta_{k},\delta_{k}\pm\varepsilon_{i}: 1\leq i\neq j \leq m , 2\leq k \neq l \leq n \}\cup\{2\delta_{1}\},\\
\Delta_{2}&=\{\delta_{1}\pm\delta_{l},\delta_{1}\pm\varepsilon_{i}: 1\leq i\leq m ,2\leq l \leq n \}.
 \end{align*}
Similarly, for either $m-2\leqslant n$ or $ n \geqslant m$ the splinter ($D(m-1,n), (2m-2)A_{1}+2nA(0,0)$) of $D(m,n)$  is given by
 \begin{align*}
 \Delta_{1}&=\{\varepsilon_{i}\pm\varepsilon_{j},\delta_{k}\pm\delta_{l},2\delta_{k},\delta_{k}\pm\varepsilon_{i}: 2\leq i\neq j \leq m , 1\leq k \neq l \leq n \},\\
\Delta_{2}&=\{\varepsilon_{1}\pm\varepsilon_{i},\delta_{1}\pm\varepsilon_{k}: 2\leq i\leq m ,1\leq l \leq n \}.
 \end{align*}
\item The splinter ($D_{m}+C{n},2mnA(0,0)$) of $D(m,n)$ is given by
 \begin{align*}
 \Delta_{1}&=\{\varepsilon_{i}\pm\varepsilon_{j},\delta_{k}\pm\delta_{l},2\delta_{k}: 1\leq i\neq j \leq m , 1\leq k \neq l \leq n \},\\
\Delta_{2}&=\{\delta_{i}\pm\varepsilon_{k}: 1\leq i\leq m ,1\leq k \leq n \}.
 \end{align*}
\end{enumerate}

\begin {table}[h]
\caption {$D(m,n)$} \label{tab:title} 
\begin{center}

\begin{tabular}{ | b{3cm} | m{6cm}| m{6cm} | } 
 \hline 
  $\Delta$     & $\Delta_{1}$ & $\Delta_{2}$ \\ \cline{1-3} 
  $D(2,1)$ & $A(1,0)$ & $2A_{1}+2A(0,0)$ \\ 
              
\hline   
  $D(2,2)$ & $A(1,1)$ & $4A_{1}+A(0,0)$ \\

    \hline
\end{tabular}
\end{center}
\end {table} 
     
     \begin {table}[h]
\begin{center}

\begin{tabular}{ | b{3cm} | m{6cm}| m{6cm} | } 
  
  \hline
     $D(3,1)$ & $A(2,0)$ & $4A_{1}+3A(0,0)$ \\                   
\hline         
$D(2,3)$ & $A(1,2)2+A_{1}$ & $5A_{1}+6A(0,0)$ \\ 
               
\hline$D(3,2)$ & $A(2,1)+A_{1}$ & $5A_{1}+6A(0,0)$ \\ 
              
\hline 
$D(m,n)$ & $D_{m}+C{n}$ & $2mnA(0,0)$ \\
\hline
$D(m,n)$ for either $n-2\leqslant m$ or $m\geqslant n$ & $D(m,n-1)+A_{1}$ & $(2n-2)A_{1}+2mA(0,0)$\\
\hline
$D(m,n)$ for either $m-2\leqslant n$ or $ n \geqslant m$ & $D(m-1,n)$ & $(2m-2)A_{1}+2nA(0,0)$\\
\hline
\end{tabular}
\end{center}
\end {table}
\section{Splints of Lie superalgebras $G(3),F(4),D(2,1;\alpha)$}
\begin{enumerate}
\item The positive root system $G(3)$ is given by 
\[\Delta=\{2\delta,\varepsilon_{i},\varepsilon_{i}-\varepsilon_{j},\delta,\varepsilon_{i}\pm\delta: 1\leq i\neq j \leq 3~and~ \varepsilon_{1}+\varepsilon_{2}+\varepsilon_{3}=0 \},\]
with the normalization
\[(\varepsilon_{i},\varepsilon_{j})=-3\delta_{ij}+1,~~(\delta,\delta)=2,~~(\varepsilon_{i},\delta)=0 ~~~~~ for ~~1\leq i,j \leq 3 . \]
Where the distinguished simple root system is given by;
\[\alpha_{1}=\delta+\varepsilon_{3},\alpha_{2}=\varepsilon_{1},\alpha_{3}=\varepsilon_{2}-\varepsilon_{1}.\]

The root system $G(3)$ has two splints and the splints are $(A_{2}+3A(0,0),A_{2}+A_{1}+4A(0,0))$ and $(B_{2}+A_{1}+3A(0,0),2A_{1}+4A(0,0)),$ these are given by 
 \begin{align*}
 \Delta_{1}&= \{\alpha_{2},\alpha_{2}+\alpha_{3},2\alpha_{2}+\alpha_{3}\} \cup \{\alpha_{1}+\alpha_{2},\alpha_{1}+\alpha_{2}+\alpha_{3},\alpha_{1}+3\alpha_{2}+\alpha_{3}\}, \\ 
\Delta_{2}&= \{\alpha_{3},3\alpha_{2}+\alpha_{3},3\alpha_{2}+2\alpha_{3}\} \cup \{\alpha_{1},\alpha_{1}+2\alpha_{2}+\alpha_{3},\alpha_{1}+3\alpha_{2}+2\alpha_{3},\alpha_{1}+4\alpha_{2}+2\alpha_{3}\} \cup \{2\alpha_{1}+4\alpha_{2}+2\alpha_{3}\}. 
 \end{align*}
And another one is
 \begin{align*}
 \Delta_{1}&= \{\alpha_{2},\alpha_{3},\alpha_{2}+\alpha_{3},2\alpha_{2}+\alpha_{3}\} \cup \{2\alpha_{1}+4\alpha_{2}+2\alpha_{3}\} \cup  \{\alpha_{1},\alpha_{1}+\alpha_{2},\alpha_{1}+\alpha_{2}+\alpha_{3}\},\\
\Delta_{2}&= \{3\alpha_{2}+\alpha_{3},3\alpha_{2}+2\alpha_{3}\} \cup \{\alpha_{1}+2\alpha_{2}+\alpha_{3},\alpha_{1}+3\alpha_{2}+2\alpha_{3},\alpha_{1}+4\alpha_{2}+2\alpha_{3},\alpha_{1}+3\alpha_{2}+\alpha_{3}\}.
 \end{align*}
\item The positive root system of $F(4)$ is given by 
\[\Delta=\{\delta,\varepsilon_{i}\pm\varepsilon_{j},\varepsilon_{i},
\frac{1}{2}(\varepsilon_{1}\pm\varepsilon_{2}\pm\varepsilon_{3}\pm\delta): 1\leq i\neq j \leq 3\}\]
with the normalization
\[(\varepsilon_{i},\varepsilon_{j})=-2\delta_{ij},~~(\delta,\delta)=6,~~(\varepsilon_{i},\delta)=0 ~~~~~ for ~~1\leq i,j \leq 3  .\]
Where the distinguished simple root system is given by;
\[\alpha_{1}=\frac{1}{2}(\delta-\varepsilon_{1}-\varepsilon_{2}-\varepsilon_{3}),\alpha_{2}=\varepsilon_{3},\alpha_{3}=\varepsilon_{2}-\varepsilon_{3},\alpha_{4}=\varepsilon_{1}-\varepsilon_{2}.\]
 In the root system of $F(4)$ we can observe that the root system of $C(4)$,$B(0,3)$,$B(0,2)$ are not embedded in $F(4)$. Only the root system of $D(2,1)$ is embedded in $F(4)$ which is identified as \[\{\alpha_{2}+\alpha_{3},2\alpha_{2}+\alpha_{3}+\alpha_{4},2\alpha_{1}+3\alpha_{2}+2\alpha_{3}+\alpha_{4}\}\cup\{\alpha_{1},\alpha_{1}+\alpha_{2}+\alpha_{3},\alpha_{1}+2\alpha_{2}+\alpha_{3}+\alpha_{4},\alpha_{1}+3\alpha_{2}+2\alpha_{3}+\alpha_{4}\}.\] So $A(2,1)$ is also embedded in $F(4)$. Hence the root system $F(4)$ has only one splint   \[(A_{1}+B_{3},8A(0,0)),\]
 where $\Delta_{1}$ and $\Delta_{2}$ are all even roots and odd roots respectively.
\item   The positive root system of $D(2,1;\alpha)$ is given by 
\[\Delta=\{2\varepsilon_{i},
(\varepsilon_{1}\pm\varepsilon_{2}\pm\varepsilon_{3}): 1\leq i \leq 3\},\]
with the normalization
\[(\varepsilon_{1},\varepsilon_{1})=-\dfrac{(1+\alpha)}{2},~~(\varepsilon_{2},\varepsilon_{2})=-\dfrac{1}{2},~~(\varepsilon_{3},\varepsilon_{3})=-\dfrac{\alpha}{2},~~(\varepsilon_{i},\varepsilon_{j})=0 ~~for~~ i\neq j  . \]
Where the distinguished simple root system is given by;
\[\alpha_{1}=\varepsilon_{2}-\varepsilon_{1},\alpha_{2}=2\varepsilon_{2},\alpha_{3}=2\varepsilon_{3}.\]
In a similar way the root system $D(2,1;\alpha)$ has two splinters that are $(A(1,0)+A_{1},A_{1}+2A(0,0))$ and $(3A_{1},4A(0,0))$ and these are given by  \begin{align*}
\Delta_{1}&=\{\alpha_{1},\alpha_{2},\alpha_{1}+\alpha_{2}\}\cup \{\alpha_{3}\},\\ \Delta_{2}&=\{2\alpha_{1}+\alpha_{2}+\alpha_{3}\}\cup \{\alpha_{1}+\alpha_{3},\alpha_{1}+\alpha_{2}+\alpha_{3}\} 
\end{align*} 
and  \begin{align*}
\Delta_{1}&=\{\alpha_{1},\alpha_{3},2\alpha_{1}+\alpha_{2}+\alpha_{3}\},\\
\Delta_{2}&=\{\alpha_{1},\alpha_{1}+\alpha_{2},\alpha_{1}+\alpha_{3},\alpha_{1}+\alpha_{2}+\alpha_{3}\}
\end{align*}
 respectively. 
 \end{enumerate}
 \section{Concluding remarks}
 In this paper we have determined splints of all classical Lie superalgebras up to equivalence with Weyl group of the corresponding algebra. We hope results of this paper can help us to some extent in determining the branching coefficient. We want to delve in to this aspect of research in future.
\section{Acknowledgement}
 One of the author prof. K.C.Pati thank National Board of Higher Mathematics(DAE), India for the Project Grant No. 2$\mid$48(25)$\mid$2016 R\&DII/4341 dt: 27.03.17 .

{} 
\end{document}